\documentclass[preprint,12pt,3p]{elsarticle}
\usepackage{amsmath,amsthm,amsfonts,amssymb}

\newtheorem{definition}{Definition}
\newtheorem{theorem}{Theorem}
\newtheorem{lemma}{Lemma}
\newtheorem{corollary}{Corollary}
\newtheorem{remark}{Remark}
\newtheorem{example}{Example}
\usepackage{amsmath}
\usepackage{amssymb}
\usepackage{graphicx}
 \usepackage{colortbl}
 \usepackage[dvipsnames]{xcolor}
\definecolor{warmblack}{rgb}{0.0, 0.26, 0.26}
\usepackage{float}
\usepackage[caption = false]{subfig}

\begin{document}

\begin{frontmatter}
		\title{  Some spectral properties and convergence of the $ (A,q)$-numerical radius and $ (A,q)$-Crawford number }

\author[label1]{Pembe Ipek Al}\ead{ipekpembe@gmail.com}
			
			 \address[label1]{Department of Mathematics, Faculty of Sciences, Karadeniz Technical University,
Trabzon, 61080 Turkey}
\author[label1]{Zameddin I. Ismailov}\ead{ zameddin.ismailov@gmail.com}
			
			% \address[label1]{}

			 \author[label3]{Fuad Kittaneh}\ead{fkitt@ju.edu.jo}
			
			 \address[label3]{Department of Mathematics, The University of Jordan, Amman, Jordan}
		
		\author[label4]{Satyajit Sahoo}\ead{satyajitsahoo2010@gmail.com, ssahoomath@gmail.com }

		\address[label4]{School of Mathematical Sciences, National Institute of Science Education and Research, Bhubaneswar, India}

		%\cortext[cor2]{Corresponding author}

		\begin{abstract}
							In this study, some estimates are given for the $ (A,q)$-numerical radius and $ (A,q)$-Crawford number via the $ A$-numerical radius and $ A$-Crawford number for the $ A $-bounded linear operators in any complex semi-Hilbert space, respectively. Then, some evolutions are studied for the tensor product of two operators. Lastly, some convergence properties of the $ (A,q)$-numerical radius and $ (A,q)$-Crawford number, via the $ A$-uniform convergence of operator sequences, are investigated. We also considered several examples to illustrate our results. Finally, a few applications of some operator functions classes are also given.
		\end{abstract}
		
    
	\end{frontmatter}

\noindent

\setcounter{section}{1}
\section*{\normalsize{1. Introduction}}
Throughout this article, $ H $ denotes a complex Hilbert space endowed with the inner product $ \langle \cdot , \cdot \rangle $ and associated norm $ \Vert \cdot \Vert . $ Let $ L(H) $ stand for the $ C^{*} $-algebra of all bounded linear operators acting on $ H. $ The identity operator on $ H $ will be simply denoted by $ I .$ From now on, by an operator we mean a bounded linear operator acting on $ H. $ Recall that an operator $ T $ is called positive, written $ T\geq 0, $ if $ \langle Tx,x \rangle\geq 0 $ for all $ x\in H. $ The square root of a positive operator $ T $ is denoted by $ T^{1/2}. $ We denote by $ \vert T \vert= (T^{*}T)^{1/2} $ the absolute value of an operator $ T. $ Here, $ T^{*} $ denote the adjoint of $ T. $ The Crawford number of an operator $ T $ is given by
$$
c(T)=\inf \lbrace \vert\langle Tx,x\rangle\vert:x\in H, \ \Vert x \Vert=1 \rbrace .
$$
The usual operator norm and the numerical radius of an operator $ T $ are, respectively, defined by
$$
\Vert T \Vert=\sup \lbrace \Vert Tx \Vert: \ x\in H, \  \Vert x \Vert=1 \rbrace
$$
and 
$$
\omega (T) =\sup \lbrace \vert\langle Tx,x\rangle\vert: \ x\in H, \ \Vert x \Vert=1 \rbrace .
$$
An operator $ T $ is called normal if $ T^{*}T=TT^{*}. $ It is well known that the equality $ \Vert T \Vert= \omega (T) $ holds for every normal operator $ T. $ However, in general, the above equality fails to be true for non-normal operators. Notice that the following inequalities hold for every $ T\in L(H): $
\begin{eqnarray}
\label{equ1}
\frac{1}{2}\Vert T \Vert \leq \omega (T) \leq \Vert T \Vert .
\end{eqnarray}
The inequalities (\ref{equ1}) play an important role in the approximation of $ \omega (\cdot ) $, which have been the interest of numerous authors. The reader may consult, for example, \cite{bhunia1, bhunia2, bhunia3, heydarbeygi, kittaneh, qiao} and the references therein or the  books \cite{bhunia4,GAu,KEGustafsonDKMRao1991}. \\
For the rest of this paper, $ A\neq 0 $ will denote a bounded linear positive operator on $ H. $ We are going to consider an additional semi-inner product $ \langle \cdot , \cdot \rangle_{A} $ on $ H $ defined by $ \langle x,y \rangle_{A}=\langle Ax, y \rangle $ for all $ x, y \in H, $ which induces a semi-norm  $ \Vert x \Vert_{A} $ on $ H. $ This makes $ H $ into a semi-Hilbert space. After that, we replace the operator norm with the following operator semi-norm
$$
\Vert T \Vert_{A}=\sup \lbrace \Vert Tx \Vert_{A}: \ x\in H, \ \Vert x \Vert_{A}=1 \rbrace 
$$
for an $ A $-bounded operator $ T, $ i. e. an operator $ T $ which satisfies $ \Vert Tx \Vert_{A}\leq c \Vert x \Vert_{A} $ for all $ x\in H $ and some constant $ c>0. $ Let $ L^{A}(H) $ stand for all  $ A$-bounded linear operators acting on $ H. $ \\
Recently, several generalizations for the concept of the numerical radius have been introduced (see \cite{sheybani} and references therein).  One of these generalizations is the so-called $ A- $numerical radius of an operator $ T\in L^{A}(H), $ which was firstly introduced by Saddi in \cite{saddi} as 
$$
\omega _{A}(T)=\sup \lbrace \vert \langle Tx, x\rangle_{A} \vert :x\in H, \Vert x \Vert_{A}=1 \rbrace.
$$
The related $ A $-Crawford number is defined as
$$
c _{A}(T)=\inf \lbrace \vert \langle Tx, x\rangle_{A} \vert :x\in H, \Vert x \Vert_{A}=1 \rbrace.
$$
Note that, it may happen that $ \omega_{A}(T)=+\infty  $ for some $ T\in L^{A}(H) $ (see \cite{baklouti}). Several results on the $ A$-numerical  radius have been established by many mathematicians. See, e.g. \cite{bhunia5, bhunia6, conde, FS, KITSAT, MMJ, PR, NSD, Nirmal2, SS, Zam} and references therein. \\
For $ q\in \mathbb{C}, \ \vert q \vert \leq 1, $ the $ q $-numerical range is defined as
$$
W_{q}(T)=\lbrace \langle Tx,y\rangle : \Vert x \Vert = \Vert y \Vert =1, \ \langle x,y \rangle =q \rbrace 
$$
(see \cite{li, marcus, rajic, stampfli}). \\
Recall that for $ T\in L(H), \  q\in \mathbb{C}, \ \vert q \vert \leq 1,  $ the $ q$-numerical radius and $ q$-Crawford number are defined as
\begin{eqnarray}
\omega_{q}(T)=\sup\lbrace \vert \lambda \vert: \lambda \in W_{q}(T) \rbrace , \ c_{q}(T)=\inf\lbrace \vert \lambda \vert: \lambda \in W_{q}(T) \rbrace , \nonumber
\end{eqnarray}
respectively. \\
Motivated by this, we introduce a new definition. 
\begin{definition}
For the operator $ T\in L^{A}(H), \ q\in \mathbb{C}, \ 0< \vert q \vert \leq 1, $ the following set
$$
W_{A,q}(T):= \left\lbrace \langle Tx,y \rangle_{A}: \ \Vert x \Vert_{A}= \Vert y \Vert_{A}=1, \ \langle x,y\rangle_{A}=q \right\rbrace 
$$
is called the $ (A,q)$-numerical range of the operator $ T $. Also, the numbers
$$
\omega _{A,q}(T)=\sup \lbrace \vert \lambda \vert : \lambda \in W_{A,q}(T) \rbrace
$$
and
$$
c_{A,q}(T)=\inf \lbrace \vert \lambda \vert : \lambda \in W_{A,q}(T) \rbrace,
$$
are called the $ (A,q)$-numerical radius and $ (A,q) $-Crawford number of the operator $ T $, respectively. \\
\end{definition}
In this case it is clear that
\begin{eqnarray}
& & W_{A,1}(T)=W_{A}(T), \ W_{I,1}(T)=W(T), \nonumber \\
& & \omega_{A,1}(T)=\omega_{A}(T), \ c_{A,1}(T)=c_{A}(T), \nonumber \\
& & \omega_{I,1}(T)=\omega(T), \ c_{I,1}(T)=c(T), \nonumber \\
& & W_{I,q}(T)=W_{q}(T), \ \omega_{I,q}(T)=\omega_{q}(T), \ c_{I,q}(T)=c_{q}(T). \nonumber 
\end{eqnarray}
Throughout this paper, the $ (A,q)$-numerical gap is the difference between the $ A $-operator semi-norm and $ (A,q) $-numerical radius, which is defined as 
$$
g_{\omega_{A,q}}(T):=\Vert T \Vert_{A} - \omega_{A,q}(T)\geq 0 \ \text{(see, Theorem \ref{thm1} (1))}
$$
for $ T\in L^{A}(H), \ q\in \mathbb{C}, \ 0< \vert q \vert \leq 1. $ \\
Also, the $ (A,q) $-Crawford gap is the difference between the $ A $-operator semi-norm and $ (A,q) $-Crawford number, which is defined as 
$$
g_{c_{A,q}}(T):=\Vert T \Vert_{A} - c_{A,q}(T)
$$
for $ T\in L^{A}(H), \ q\in \mathbb{C}, \ 0< \vert q \vert \leq 1. $ \\
In this study, the contents are as follows. In Section 2, some estimates are given for the $ (A,q) $-numerical radius and $ (A,q) $-Crawford number via the $ A $-numerical radius and $ A$-Crawford number, respectively, for the $ A $-bounded linear operators in any complex semi-Hilbert space. Then, some evolutions are studied for the tensor product of two operators. We also consider several examples on Hilbert space and establish $\omega_q(\cdot)$ only for a real number $q\in [0, 1]$ to illustrate our results. In Section 3, some convergence properties of the $ (A,q) $-numerical radius and $ (A,q) $-Crawford number via $ A $-uniformly convergence of operator sequences are investigated. In Section 4, a few applications of some operator functions classes are also given.

%%% ----------------------------------------------------------------------
\setcounter{section}{2}
\section*{\normalsize{2. Some properties of the $ (A,q)$-numerical radius and $(A,q)$-Crawford number}}
\begin{theorem}
\label{thm1}
For the operator $ T\in L^{A}(H) $ and $ q,\lambda,\mu\in \mathbb{C}, \ 0<\vert q \vert\leq 1, $ the following are true  \\
(1) $ \omega_{A,q}(T)\leq \Vert T \Vert_{A}. $ \\
(2) If  $ \alpha \in \mathbb{C}, \ \vert \alpha \vert =1, $ then $ \omega_{A,q}(\alpha T)=\omega_{A,\alpha q}(T). $ \\
(3) If  $ \alpha \in \mathbb{C}, \ \vert \alpha \vert =1,$ then $ c_{A,q}(\alpha T)=c_{A,\alpha q}(T). $ \\
(4) If $  \gamma =\sqrt{\vert \lambda\vert^{2}+ \vert \mu\vert^{2}+ 2Re (\lambda \overline{\mu} q)}\neq 0, $ then $ \gamma \omega_{A,\frac{\lambda + \mu \overline{q}}{\gamma}} (T)\leq \vert \lambda\vert\omega_{A}(T)+\vert \mu\vert\omega_{A,q}(T^{*}). $  \\
(5) If $  \gamma =\sqrt{\vert \lambda\vert^{2}+ \vert \mu\vert^{2}+ 2Re (\lambda \overline{\mu} q)}\neq 0, $ then $ \gamma c_{A,\frac{\lambda + \mu \overline{q}}{\gamma} }(T)\leq \vert \lambda\vert\omega_{A}(T)+\vert \mu\vert c_{A,q}(T^{*}). $\\
(6) For $ x,y \in H, \ \Vert x \Vert_{A}=\Vert y \Vert_{A}=1, \ \langle x,y \rangle_{A}=q, \ 0<\vert q \vert \leq 1 $, it is true $$
\Vert x\pm y \Vert_{A}=\sqrt{2}\sqrt{1\pm Req}.
$$ 
(7) $ \omega_{A}(T)\leq \omega_{A, q}(T)+ \sqrt{2(1-Req)}\omega_{A,\frac{1-q}{\sqrt{2(1-Req)}}}(T), \ q\neq 1. $ \\
(8) $ c_{A}(T)\leq c_{A, q}(T)+ \sqrt{2(1-Req)}\omega_{A,\frac{1-q}{\sqrt{2(1-Req)}}}(T), \ q\neq 1 . $
\end{theorem}
\begin{proof}
%\noindent {\it \textbf{Proof}} 
(1) For $ x, y \in H, $ $ \Vert x \Vert_{A}=\Vert y \Vert_{A}=1, \ \langle x,y \rangle_{A}=q, $ we have 
\begin{eqnarray}
\vert \langle Tx,y \rangle_{A}\vert & = & \vert \langle ATx,y \rangle \vert \nonumber \\
& \leq & \vert \langle ATx,Tx \rangle^{1/2} \vert \vert \langle Ay,y \rangle^{1/2} \vert \nonumber \\
& = & \vert \langle Tx,Tx \rangle_{A}^{1/2} \vert \vert \langle y,y \rangle_{A}^{1/2} \vert \nonumber \\
& = & \Vert Tx \Vert_{A} \Vert y \Vert_{A} \nonumber \\
& = & \Vert Tx \Vert_{A}.  \nonumber
\end{eqnarray}
Hence, it is obtained that
$$
\omega_{A,q}(T)\leq \Vert T \Vert_{A}.
$$
(2)-(3) These equalities are results of the following relations 
$$
\Vert \alpha x \Vert_{A}=\vert \alpha \vert \Vert  x \Vert_{A}=1,
$$
$$
\langle \alpha x,y \rangle_{A}=\alpha \langle x,y \rangle_{A}=\alpha q, \ 0<\vert \alpha q \vert= \vert q \vert \leq 1.
$$
(4)-(5) For the operator $ T\in L^{A}(H) $ and $ \lambda,\mu\in \mathbb{C},$ the following relation
\begin{eqnarray}
\vert \langle T(\lambda x+\mu y),x\rangle_{A} \vert &\leq & \vert \lambda\vert \vert \langle Tx,x\rangle_{A} \vert +  \vert \mu\vert \vert \langle Ty,x\rangle_{A} \vert  \nonumber \\
& = & \vert \lambda\vert \vert \langle Tx,x\rangle_{A} \vert +  \vert \mu\vert \vert \langle T^{*}x,y\rangle_{A} \vert \nonumber
\end{eqnarray}
is true. Since for $ \Vert x \Vert_{A}=\Vert y \Vert_{A}=1, \ \langle x,y\rangle_{A}=q, \ 0<\vert q \vert\leq 1, $ 
\begin{eqnarray}
\Vert \lambda x+\mu y \Vert_{A}^{2} & = & \langle A(\lambda x+\mu y), \lambda x+\mu y \rangle \nonumber \\
& = & \langle A(\lambda x), \lambda x \rangle +\langle A(\lambda x), \mu y \rangle +\langle A(\mu y), \lambda x\rangle + \langle A(\mu y), \mu y \rangle \nonumber \\
& = & \vert \lambda \vert^{2} \langle Ax,  x \rangle + \lambda \overline{\mu} \langle Ax,y \rangle + \overline{\lambda}\mu \langle Ay, x \rangle +\vert \mu \vert^{2} \langle Ay, y \rangle \nonumber \\
& = & \vert \lambda \vert^{2} \langle x,  x \rangle_{A} + \lambda \overline{\mu} \langle x,y \rangle_{A} + \overline{\lambda}\mu \langle y, x \rangle_{A} +\vert \mu \vert^{2} \langle y, y \rangle_{A} \nonumber \\
& = & \vert \lambda \vert^{2}+\vert \mu \vert^{2}+\lambda \overline{\mu}q+\overline{\lambda}\mu\overline{q}\nonumber \\
& = & \vert \lambda \vert^{2}+\vert \mu \vert^{2}+2Re\left( \lambda \overline{\mu}q \right), \nonumber
\end{eqnarray}
then for $ \gamma =\sqrt{\vert \lambda\vert^{2}+ \vert \mu\vert^{2}+ 2Re (\lambda \overline{\mu} q)} $ we have   $
\Vert \dfrac{\lambda x+\mu y}{\gamma}  \Vert_{A}=1 $   and $
\langle \frac{\lambda x+\mu y }{\gamma},x \rangle_{A}=\frac{\lambda +\overline{\mu} q }{\gamma}.$
Hence, from the above inequality and the definitions of $ A$-numerical and $ (A,q)$-numerical radii we have
$$
\gamma \omega_{A,\frac{\lambda + \mu \overline{q}}{\gamma}} (T)\leq \vert \lambda\vert\omega_{A}(T)+\vert \mu\vert\omega_{A,q}(T^{*}).
$$

Similarly, from the above inequality and the definitions of $ A$-numerical radius and $ (A,q)$-Crawford number it is implies
$$
\gamma c_{A,\frac{\lambda + \mu \overline{q}}{\gamma} }(T)\leq \vert \lambda\vert\omega_{A}(T)+\vert \mu\vert c_{A,q}(T^{*}).
$$
(6) For $ x,y \in H, \ \Vert x \Vert_{A}=\Vert y \Vert_{A}=1, \ \langle x,y \rangle_{A}=q, \ 0<\vert q \vert \leq 1 $, it is true 
\begin{eqnarray}\label{INeq_2}
\Vert x\pm y \Vert^{2}_{A} & = & \langle A(x\pm y),x\pm y \rangle \nonumber \\
& = & \langle Ax,x \rangle \pm \langle Ax,y \rangle \pm \langle Ay,x \rangle + \langle Ay,y \rangle \nonumber \\
& = & \Vert x \Vert_{A}^{2}\pm 2 Re \langle x,y \rangle_{A}+ \Vert y \Vert_{A}^{2} \nonumber \\
& = & 2 \pm 2 Req. 
\end{eqnarray}
Hence, we have
$$
\Vert x\pm y \Vert_{A}=\sqrt{2}\sqrt{1\pm Req}.
$$
(7) For any $ x,y \in H, \ \Vert x \Vert_{A}=\Vert y \Vert_{A}=1, \  \langle x,y \rangle_{A}=q, \ q\neq 1, $ we have 
\begin{eqnarray}
\vert \langle Tx,x\rangle_{A} \vert & \leq & \vert \langle Tx,y\rangle_{A} \vert +  \vert \langle Tx,x-y\rangle_{A} \vert \nonumber \\
& = & \vert \langle Tx,y\rangle_{A} \vert + \sqrt{2(1-Req)}\vert \langle Tx,\frac{x-y}{\sqrt{2(1-Req)}}\rangle_{A} \vert . \nonumber 
\end{eqnarray}
Then, from (6), we have $ \Vert \frac{x-y}{\sqrt{2(1-Req)}} \Vert_{A}=1 $ and $ \langle x, \frac{x-y}{\sqrt{2(1-Req)}} \rangle =\frac{1-q}{\sqrt{2(1-Req)}}. $
Hence, from the previous relation and the definitions of $ A$-numerical and $ (A,q)$-numerical radii we have
$$
\omega_{A}(T)\leq \omega_{A, q}(T)+ \sqrt{2(1-Req)}\omega_{A,\frac{1-q}{\sqrt{2(1-Req)}}}(T).
$$
(8) This claim can be proved similarly. \\ \\
  \end{proof}
\noindent {\it \textbf{Note}}: From the claims (7) and (8) of Theorem \ref{thm1}, if $ q=1, $ then
$$
\omega_{A,q}(T)=\omega_{A}(T), \ c_{A,q}(T)=c_{A}(T)
$$
for any operator $ T\in L^{A}(H). $ \\ \\
\noindent {\it \textbf{Note}}: From Theorem \ref{thm1}, the following is true
$$
(1-\sqrt{2(1-Req)})\omega_{A}(T)\leq \omega_{A,q}(T)
$$
for any $ T\in L^{A}(H) $ and $ q\in \mathbb{C}, \ 0<\vert q \vert\leq 1. $
\begin{theorem}\label{Thhm_2}
For the operator $ T\in L^{A}(H) $ and $ q\in \mathbb{C}, \ 0<\vert q \vert\leq 1, $ the following is true
$$
2\vert Re q \vert\omega_{A}(T) \leq \omega_{A,q}(T)+ \omega_{A,\overline{q}}(T) \leq 2 \omega_{A}(T) + 2\sqrt{2}\sqrt{1-Req} \Vert T \Vert_{A}.
$$
\end{theorem}
\begin{proof} 
For $ x,y \in H, \ \Vert x \Vert_{A}=\Vert y \Vert_{A}=1, \ \langle x,y \rangle_{A}=q, \ 0<\vert q \vert \leq 1, $  it is clear that 
\begin{eqnarray}
\langle T(x\pm y),x\pm y \rangle_{A} & = & \langle AT(x\pm y), x\pm y \rangle \nonumber \\
& = & \langle ATx,x\rangle \pm \langle ATx,y\rangle \pm \langle ATy,x\rangle + \langle ATy,y\rangle \nonumber \\
& = & \langle Tx,x\rangle_{A} \pm \langle Tx,y\rangle_{A} \pm \langle Ty,x\rangle_{A} + \langle Ty,y\rangle_{A}. \nonumber
\end{eqnarray}
Since,
\begin{eqnarray}
\vert \langle T(x\pm y),x\pm y \rangle_{A} \vert \leq \vert\langle Tx,x\rangle_{A}\vert + \vert\langle Tx,y\rangle_{A}\vert + \vert \langle Ty,x\rangle_{A}\vert + \vert\langle Ty,y\rangle_{A}\vert , \nonumber
\end{eqnarray}
then for $ -1\leq Req<0 $, we have
$$
2(1-Req)\vert \langle T \left( \frac{x-y}{\sqrt{2(1-Req)}}\right) , \frac{x-y}{\sqrt{2(1-Req)}} \rangle_{A} \vert \leq 2\omega_{A}(T)+\omega_{A,q}(T)+ \omega_{A,\overline{q}}(T) .
$$
Consequently, 
$$
2(1-Req)\omega_{A}(T)\leq 2\omega_{A}(T) +\omega_{A,q}(T)+ \omega_{A,\overline{q}}(T).
$$
That is, for $ q\in \mathbb{C}, \ 0<\vert q \vert \leq 1, \ -1\leq Req < 0, $ 
$$
-2Req\omega_{A}(T)\leq \omega_{A,q}(T)+ \omega_{A,\overline{q}}(T).
$$
Similarly, for the $ 0\leq Req \leq 1 $, we have
$$
2(1+Req)\vert \langle T \left( \frac{x+y}{\sqrt{2(1+Req)}}\right) , \frac{x+y}{\sqrt{2(1+Req)}} \rangle \vert \leq 2\omega_{A}(T)+\omega_{A,q}(T)+ \omega_{A,\overline{q}}(T) .
$$
Then
$$
2(1+Req)\omega_{A}(T)\leq 2\omega_{A}(T) +\omega_{A,q}(T)+ \omega_{A,\overline{q}}(T).
$$
Hence, we have
$$
2Req \omega_{A}(T)\leq \omega_{A,q}(T)+ \omega_{A,\overline{q}}(T).
$$
Consequently, from these facts, it is established that 
$$
2 \vert Req \vert \omega_{A}(T)\leq \omega_{A,q}(T)+ \omega_{A,\overline{q}}(T)
$$
for $ q\in \mathbb{C}, \ 0<\vert q \vert \leq 1. $ \\
On the other hand, for $ x,y \in H, \ \Vert x \Vert_{A}=\Vert y \Vert_{A}=1, \ \langle x,y \rangle_{A}=q, \ 0<\vert q \vert \leq 1, $ it is clear that
$$
\vert \langle Tx,y \rangle_{A} \vert \leq \vert \langle Tx,x \rangle_{A} \vert + \vert \langle Tx,y-x \rangle_{A} \vert .
$$
Then, we have
\begin{eqnarray}
\omega _{A,q}(T) & \leq & \omega _{A}(T)+ \Vert T \Vert_{A} \Vert x-y \Vert_{A} \nonumber \\
& \leq &  \omega _{A}(T)+ \sqrt{2(1-Req)}\Vert T \Vert_{A} . \nonumber
\end{eqnarray}
Similarly, we get
$$
\omega _{A,\overline{q}}(T)\leq \omega _{A}(T)+ \sqrt{2(1-Re\overline{q})}\Vert T \Vert_{A}.
$$
Then, from the first part of this theorem we have 
$$
\omega _{A,q}(T)+\omega _{A,\overline{q}}(T)\leq 2\omega _{A}(T)+ 2\sqrt{2}\sqrt{1-Req}\Vert T \Vert_{A}.
$$
\end{proof}
\begin{lemma}\label{lem1}\textnormal{\cite{Nakazato}} 
    Suppose $0\leq q\leq 1$ and $T\in M_2(\mathbb{C})$. Then $T$ is unitarily similar to $e^{it}\begin{pmatrix}
        \gamma &a\\
        b & \gamma
    \end{pmatrix}$ for some $0\leq t \leq 2\pi $ and $0\leq b\leq a$. Also,
    \begin{align*}
        W_q(T)=e^{it}\left\{\gamma q+r\left((c+pd)\cos(s)+i(d+pc)\sin (s)\right): 0\leq r\leq 1, 0\leq s\leq 2\pi \right\},
    \end{align*}
   with $c=\frac{a+b}{2}, d=\frac{a-b}{2}~\mbox{and}~ p=\sqrt{1-q^2}. $ 
\end{lemma}
\begin{remark}\label{rem0}
   Setting $A=I$ in Theorem \ref{Thhm_2}, and consider $q\in [0, 1]$,  we have 
\begin{align}\label{Eqn_2}
  2 q \omega(T) \leq 2\omega_{q}(T) \leq 2 \omega(T) + 2\sqrt{2}\sqrt{1-q} \Vert T \Vert. 
\end{align}  
\end{remark}
\begin{example}\label{Ex_1}
    Let $T=\begin{pmatrix}
        0 & \frac{1}{70}\\
        0 & 0
    \end{pmatrix}$ and $q\in [0, 1]$. Using Lemma \ref{lem1}, we get 
    $$W_q(T)=\left\{\frac{re^{is}}{140}(1+\sqrt{1-q^2}): 0\leq r\leq 1, 0\leq s\leq 2\pi\right\}.$$
    So, $$\omega_q(T)=\frac{1}{140}(1+\sqrt{1-q^2}).$$
    Since $T^2=0$,  we have $\omega(T)=\frac{\|T\|}{2}=\frac{1}{140}$. From the Remark \ref{rem0}, we get  $2q\omega(T)= \frac{q}{70}$, $2\omega_{q}(T)=\frac{1}{70}(1+\sqrt{1-q^2})$, and using the fact $\|T\|=\frac{1}{70}$, the value of $2 \omega(T) + 2\sqrt{2}\sqrt{1-q} \Vert T \Vert= \frac{1}{70}(1+2\sqrt{2}\sqrt{1-q^2}).$
     Figure~\ref{Fig-1} compares the upper and lower bounds of $2\omega_q(T)$.
    \begin{figure}[H]
\centering
{\includegraphics[scale=.58]{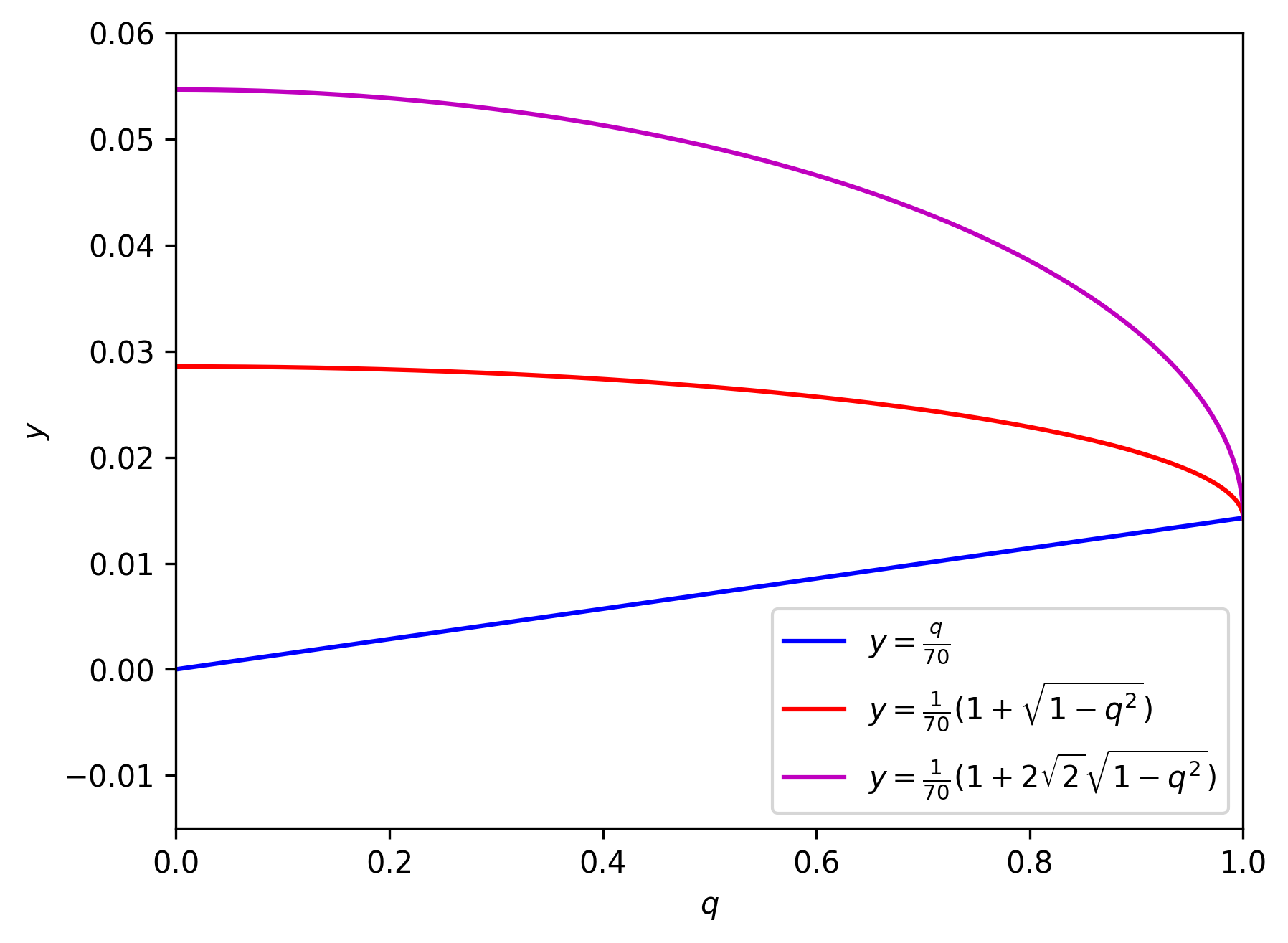}}
\caption{Comparision of $\omega_q(T) $ with the upper and lower bounds \eqref{Eqn_2} for Example \ref{Ex_1}. }	 
\label{Fig-1}
\end{figure}
   \end{example}

\begin{theorem}
\label{thm3}
If $ A_{1}\in L(H_{1}), \ A_{2}\in L(H_{2}), \ A_{1}\geq 0, \ A_{2}\geq 0 $ and $ T_{1}\in L^{A_{1}}(H_{1}), \ T_{2}\in L^{A_{2}}(H_{2}).  $ Then, for $ A=A_{1}\otimes A_{2}, \  T=T_{1}\otimes T_{2}, $ we have
$$
c_{A,q}(T)\leq c_{A_{1},q_{1}}(T_{1})c_{A_{2},q_{2}}(T_{2}) \leq \omega_{A_{1},q_{1}}(T_{1})\omega_{A_{2},q_{2}}(T_{2})\leq \omega_{A,q}(T),
$$
where, $ q_{1},q_{2}\in \mathbb{C}, \ 0<\vert q_{1} \vert \leq 1, \ 0<\vert q_{2} \vert \leq 1, \ q=q_{1}q_{2}.  $
\end{theorem}
\begin{proof}
    For any $ x=x_{1}\otimes x_{2}, \ y=y_{1}\otimes y_{2}\in H_{1}\otimes H_{2}, $ we have $ \langle x,y \rangle_{A}= \langle x_{1},y_{1} \rangle_{A_{1}} \langle x_{2},y_{2} \rangle_{A_{2}}=q_{1}q_{2},$ $ \Vert x \Vert_{A}= \Vert x_{1} \Vert_{A_{1}}\Vert x_{2} \Vert_{A_{2}}, \ \Vert y \Vert_{A}= \Vert y_{1} \Vert_{A_{1}}\Vert y_{2} \Vert_{A_{2}} $ and
\begin{eqnarray}
\langle Tx,y \rangle_{A} & = & \langle (A_{1}\otimes A_{2})(T_{1}x_{1}\otimes T_{2}x_{2}), y_{1}\otimes y_{2}\rangle \nonumber \\
& = & \langle A_{1}T_{1}x_{1}\otimes A_{2}T_{2}x_{2}, y_{1}\otimes y_{2}\rangle \nonumber \\
& = & \langle T_{1}x_{1}, y_{1}\rangle_{A_{1}} \langle T_{2}x_{2}, y_{2}\rangle_{A_{2}} .\nonumber 
\end{eqnarray}
If $ \Vert x_{1} \Vert_{A_{1}} = \Vert x_{2} \Vert_{A_{2}} = \Vert y_{1} \Vert_{A_{1}} = \Vert y_{2} \Vert_{A_{2}}=1, $
then for $ \Vert x \Vert_{A}=\Vert y \Vert_{A}=1 $ and $ \langle x,y \rangle_{A}=q, \ q=q_{1}q_{2} $, we have 
$$
W_{A_{1},q_{1}}(T_{1})W_{A_{2},q_{2}}(T_{2})\subset W_{A,q}(T).
$$
On the other hand, since
\begin{eqnarray}
& & \sup \left\lbrace  \vert \langle T_{1}x_{1}, y_{1}\rangle_{A_{1}} \vert \vert \langle T_{2}x_{2}, y_{2}\rangle_{A_{2}} \vert: \ \Vert x_{1} \Vert_{A_{1}}=\Vert y_{1} \Vert_{A_{1}}=1, \ \langle x_{1},y_{1}\rangle_{A_{1}} =q_{1}, \right.  \nonumber \\
& & \ \ \ \ \ \ \left.  \Vert x_{2} \Vert_{A_{2}}=\Vert y_{2} \Vert_{A_{2}}=1, \ \langle x_{2},y_{2}\rangle_{A_{2}} =q_{2}  \right\rbrace   \nonumber \\
& = & \sup \left\lbrace  \vert \langle Tx, y\rangle_{A} \vert : \ \Vert x_{1} \Vert_{A_{1}}=\Vert y_{1} \Vert_{A_{1}}=1, \ \langle x_{1},y_{1}\rangle_{A_{1}} =q_{1}, \right.  \nonumber \\
& & \ \ \ \ \ \ \left.  \Vert x_{2} \Vert_{A_{2}}=\Vert y_{2} \Vert_{A_{2}}=1, \ \langle x_{2},y_{2}\rangle_{A_{2}} =q_{2}  \right\rbrace   \nonumber \\
& \leq &  \sup \left\lbrace  \vert \langle Tx, y\rangle_{A} \vert : \ \Vert x \Vert_{A}=\Vert y \Vert_{A}=1, \ \langle x,y\rangle_{A} =q, \ q=q_{1}q_{2} \right\rbrace   \nonumber 
\end{eqnarray}
and
\begin{eqnarray}
& & \inf \left\lbrace  \vert \langle T_{1}x_{1}, y_{1}\rangle_{A_{1}} \vert \vert \langle T_{2}x_{2}, y_{2}\rangle_{A_{2}} \vert: \ \Vert x_{1} \Vert_{A_{1}}=\Vert y_{1} \Vert_{A_{1}}=1, \ \langle x_{1},y_{1}\rangle_{A_{1}} =q_{1}, \right.  \nonumber \\
& & \ \ \ \ \ \ \left.  \Vert x_{2} \Vert_{A_{2}}=\Vert y_{2} \Vert_{A_{2}}=1, \ \langle x_{2},y_{2}\rangle_{A_{2}} =q_{2}  \right\rbrace   \nonumber \\
& = & \inf \left\lbrace  \vert \langle Tx, y\rangle_{A} \vert : \ \Vert x_{1} \Vert_{A_{1}}=\Vert y_{1} \Vert_{A_{1}}=1, \ \langle x_{1},y_{1}\rangle_{A_{1}} =q_{1}, \right.  \nonumber \\
& & \ \ \ \ \ \ \left.  \Vert x_{2} \Vert_{A_{2}}=\Vert y_{2} \Vert_{A_{2}}=1, \ \langle x_{2},y_{2}\rangle_{A_{2}} =q_{2}  \right\rbrace   \nonumber \\
& \geq &  \inf \left\lbrace  \vert \langle Tx, y\rangle_{A} \vert : \ \Vert x \Vert_{A}=\Vert y \Vert_{A}=1, \ \langle x,y\rangle_{A} =q, \ q=q_{1}q_{2} \right\rbrace ,  \nonumber 
\end{eqnarray}
then we get
$$
\omega_{A_{1},q_{1}}(T_{1})\omega_{A_{2},q_{2}}(T_{2})\leq \omega_{A,q}(T),
$$
$$
c_{A_{1},q_{1}}(T_{1})c_{A_{2},q_{2}}(T_{2})\geq c_{A,q}(T).
$$
\end{proof}
Using an argument similar to that used in the proof of the previous theorem, it is easy to prove the following result.
\begin{corollary}
(1) If $ c_{A_{1},q_{1}}(T_{1})>0 ,$ then $ W_{A_{2},q_{2}}(T_{2})\subset W_{A,q}(T)W_{A_{1},q_{1}}^{-1}(T_{1}). $ \\
(2) If $ c_{A_{2},q_{2}}(T_{2})>0 $, then $ W_{A_{1},q_{1}}(T_{1})\subset W_{A,q}(T)W_{A_{2},q_{2}}^{-1}(T_{2}) $. \\
Consequently, \\
if $c_{A_{1},q_{1}}(T_{1})>0, $ then $ \omega_{A_{2},q_{2}}(T_{2})\leq \frac{\omega_{A,q}(T)}{c_{A_{1},q_{1}}(T_{1})}, $  \\
if $c_{A_{2},q_{2}}(T_{2})>0, $ then $ \omega_{A_{1},q_{1}}(T_{1})\leq \frac{\omega_{A,q}(T)}{c_{A_{2},q_{2}}(T_{2})}, $\\
if $\omega_{A_{1},q_{1}}(T_{1})>0 , $ then $ c_{A_{2},q_{2}}(T_{2})\leq \frac{c_{A,q}(T)}{\omega_{A_{1},q_{1}}(T_{1})}, $\\
if $\omega_{A_{2},q_{2}}(T_{2})>0, $ then $ c_{A_{1},q_{1}}(T_{1})\leq \frac{c_{A,q}(T)}{\omega_{A_{2},q_{2}}(T_{2})}. $\\
Here, for $ S=\lbrace z\in \mathbb{C}: z\neq 0 \rbrace , \ S^{-1}= \left\lbrace \dfrac{1}{z}: z\in S  \right\rbrace .  $
\end{corollary}

%%% ----------------------------------------------------------------------
\setcounter{section}{3}
\section*{\normalsize{3. Some convergence properties of the $ (A,q)$-numerical radius and $ (A,q)$-Crawford number via operator sequences}}

\begin{definition}
A sequence $ (T_{n}) \subset L^{A}(H) $ is said to converge uniformly to $ T\in L^{A}(H), $ if for any $ \epsilon >0 , $ there exists a positive integer $ N $ such that 
$$
\Vert T_{n}- T \Vert_{A} <\epsilon \ \text{ for all} \  n>N .
$$
\end{definition}

\begin{theorem}
\label{thm4}
For any $ T\in L^{A}(H) $ and $ q\in \mathbb{C}, \ 0<\vert q \vert\leq 1 $, the following are true: \\
(1) $ \vert \omega_{A,q}(T)-\omega_{A}(T)\vert\leq \sqrt{2(1-Req)}\Vert T \Vert_{A}. $ \\
(2) If $ q_{n}\in \mathbb{C}, \ 0<\vert q_{n} \vert \leq 1 $ and $ Req_{n}\rightarrow 1, \ n\rightarrow \infty, $ then 
$$
\lim\limits_{n\rightarrow \infty}\omega_{A,q_{n}}(T)=\omega_{A}(T).
$$
(3) If the operator sequences $ (T_{n})$ converges uniformly to an operator $ T $ in $ L^{A}(H), $ then
$$
\lim\limits_{n\rightarrow \infty}\omega_{A,q}(T_{n})=\omega_{A,q}(T).
$$
\end{theorem}
\begin{proof}
   (1) For any $ x,y\in H $ with the properties $ \Vert x \Vert_{A}=\Vert y \Vert_{A}=1, \  \langle x,y \rangle_{A}=q, \ 0<\vert q \vert \leq 1, $ it is valid that
\begin{eqnarray}
\vert \langle Tx,y \rangle_{A}\vert & = & \vert \langle ATx,y \rangle \vert \nonumber \\
& \leq & \vert \langle ATx,x \rangle \vert + \vert \langle ATx,y-x \rangle \vert \nonumber \\
& \leq & \vert \langle Tx,x \rangle_{A} \vert + \langle ATx,Tx \rangle^{1/2}\langle A(y-x),y-x \rangle^{1/2} \nonumber \\
& = & \vert \langle Tx,x \rangle_{A} \vert + \Vert Tx \Vert_{A} \Vert y-x \Vert_{A}. \nonumber
\end{eqnarray}
%Also, we have
% \begin{eqnarray}
% \Vert y-x \Vert_{A}^{2} & = & \langle Ay-Ax,y-x \rangle \nonumber \\
% & = & \langle Ay,y \rangle - 2Re\langle Ax,y \rangle + \langle Ax,x \rangle \nonumber \\
% & = & \Vert y\Vert_{A}^{2} - 2Re\langle Ax,y \rangle + \Vert x\Vert_{A}^{2} \nonumber \\
% & = & 2-2Re\langle x,y \rangle_{A} \nonumber \\
% & = & 2(1-Req). \nonumber
% \end{eqnarray}
Now, using inequality \eqref{INeq_2} and the previous estimate, we get
$$
\vert \langle Tx,y \rangle_{A}\vert \leq \vert \langle Tx,x \rangle_{A}\vert + \sqrt{2}\sqrt{1-Req} \Vert T \Vert_{A}.
$$
This means that 
$$
\omega_{A,q}(T)\leq \omega_{A}(T)+ \sqrt{2}\sqrt{1-Req} \Vert T \Vert_{A}. 
$$
On the other hand, similarly from the following inequality,
$$
\vert \langle Tx,x \rangle_{A} \vert \leq \vert \langle Tx,y \rangle_{A} \vert + \vert \langle Tx,x-y \rangle_{A} \vert,
$$
it can be easily to shown that
$$
\omega_{A}(T)\leq \omega_{A,q}(T)+ \sqrt{2}\sqrt{1-Req} \Vert T \Vert_{A}. 
$$
Consequently, it is established that 
$$ 
\vert \omega_{A,q}(T)-\omega_{A}(T)\vert\leq \sqrt{2(1-Req)}\Vert T \Vert_{A}.
$$
(2) This is a result of the property of the first proposition of this theorem. \\
(3) From the definition of the $ (A,q) $-numerical radius, we have
$$
\omega _{A,q}(T+S)\leq \omega _{A,q}(T)+ \omega _{A,q}(S)
$$
for $ T,S \in L^{A}(H) $ and $ q\in \mathbb{C}, \ 0< \vert q \vert \leq 1. $ So, the following claim is true
$$
\vert \omega _{A,q}(T)-\omega _{A,q}(S) \vert \leq \omega _{A,q}(T\pm S).
$$
Hence, for any $ n\geq 1 $, from the claim (1) of Theorem \ref{thm1}, it is obtained that 
$$
\vert \omega _{A,q}(T_{n})-\omega _{A,q}(T) \vert \leq \omega _{A,q}(T_{n}-T)\leq \Vert T_{n}-T \Vert_{A}.
$$
From the last relation and the uniform convergence of the sequence $ (T_{n}) $ to $ T $ in $ L^{A}(H), $ it follows that 
$$
\lim\limits_{n\rightarrow \infty}\omega_{A,q}(T_{n})=\omega_{A,q}(T).
$$
\end{proof}
\begin{remark}\label{rem1}
   Setting $A=I$ in Theorem \ref{thm4} (1), and considering $q\in [0, 1]$,  we have 
   \begin{align}
   \label{equ2}
       \vert \omega_{q}(T)-\omega(T)\vert\leq \sqrt{2(1-q)}\Vert T \Vert,~ \mbox{for any}~  T\in L(H). 
   \end{align}
 \end{remark}
Here we consider some examples only for a real number $q\in [0, 1]$.
\begin{example}\label{Ex_2}
    Let $T=\begin{pmatrix}
        0 & \frac{1}{24}\\
        0 & 0
    \end{pmatrix}$ and $q\in [0, 1]$. Using Lemma \ref{lem1}, we get 
    $$W_q(T)=\left\{\frac{re^{is}}{48}(1+\sqrt{1-q^2}): 0\leq r\leq 1, 0\leq s\leq 2\pi\right\}.$$
    So, $$\omega_q(T)=\frac{1}{48}(1+\sqrt{1-q^2}).$$
    By Remark \ref{rem1} and the fact that $\|T\|=\frac{1}{24}$, we see that the right side of Remark \ref{rem1} is $\sqrt{2(1-q)}\Vert T \Vert= \frac{\sqrt{2(1-q)}}{24}$.
    Since $T^2=0$,  we have $\omega(T)=\frac{\|T\|}{2}=\frac{1}{48}$. Therefore, the left side of Remark \ref{rem1} is
    $ \vert \omega_{q}(T)-\omega(T)\vert=\frac{\sqrt{1-q^2}}{48}.$ Figure \ref{Fig-2} compares  $ \vert \omega_{q}(T)-\omega(T)\vert = \frac{\sqrt{1-q^2}}{48}$ and the upper bound $\frac{\sqrt{2(1-q)}}{24}$.  
  \begin{figure}[H]
\centering
{\includegraphics[scale=.58]{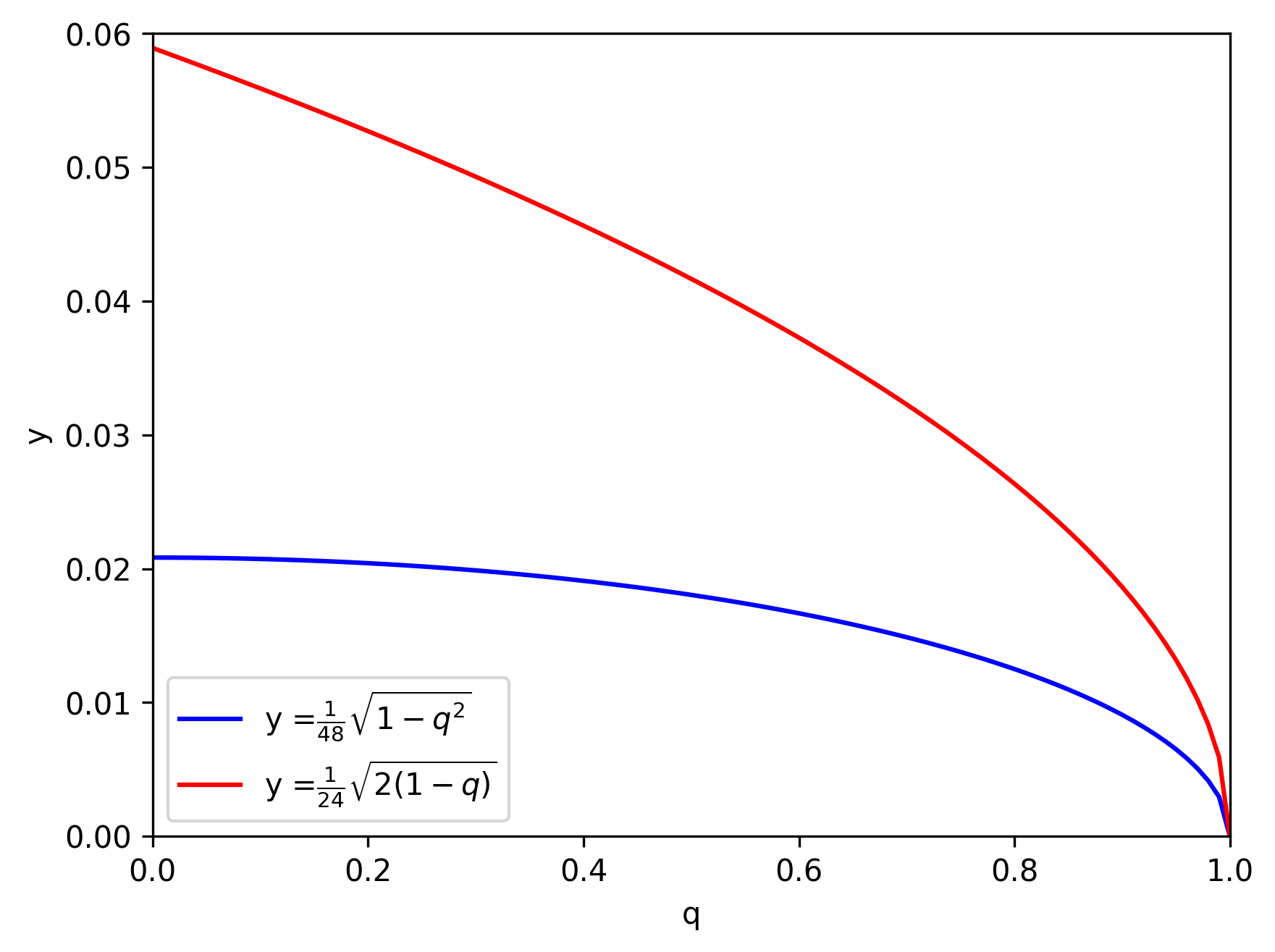}}
\caption{Comparision of $\omega_q(T) $ with the upper bound \eqref{equ2} for Example \ref{Ex_2}. }	 
\label{Fig-2}
\end{figure}
\end{example}

 \begin{example}\label{Ex_3}
    Let us consider $T=\frac{1}{20}I$ and $q\in [0, 1]$. In this case,
    $\omega_q(T)=\frac{1}{20}q.$ 
       Using $ \omega(T)=\|T\|=\frac{1}{20},$ the left side of Remark \ref{rem1} is $ \vert \omega_{q}(T)-\omega(T)\vert = \frac{1}{20}(q-1)$, while the right side of Remark \ref{rem1} is $\sqrt{2(1-q)}\Vert T \Vert=\frac{\sqrt{2(1-q)}}{20}$. Figure \ref{Fig-3} compares  $ \vert \omega_{q}(T)-\omega(T)\vert = \frac{1}{20}(q-1)$ and the upper bound $ \frac{\sqrt{2(1-q)}}{20}$.  
     \begin{figure}[H]
\centering
{\includegraphics[scale=.58]{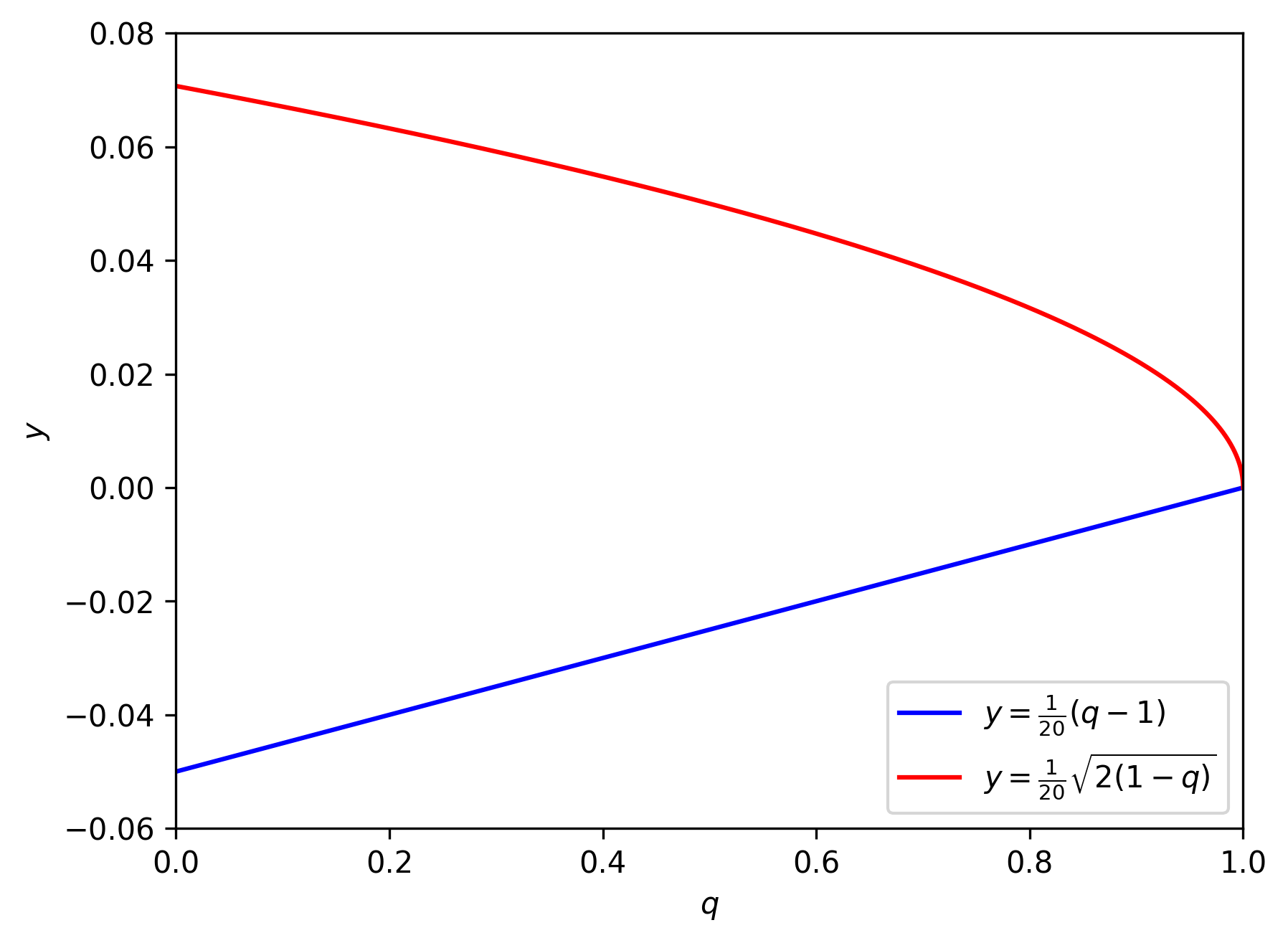}}
\caption{Comparision of $\omega_q(T) $ with the upper bound \eqref{equ2} for Example \ref{Ex_3}. }	 
\label{Fig-3}
\end{figure}
      \end{example}

\begin{example}\label{Ex_4}
    Let $T=\begin{pmatrix}
        0 &1 &0\\
        0 & 0 &1\\
        0 &0 &0
    \end{pmatrix}$ and $q\in [\frac{1}{2}, 1]$. Using \cite[Proposition 4.3]{LiNakazato}, we get 
    $$\omega_q(T)=\frac{1}{8}\sqrt{27+18q-13q^2+(9+7q)\sqrt{(1-q)(9+7q)}}.$$
    By Remark \ref{rem1} and the fact that $\|T\|=1$, we see that the right side of inequality (\ref{equ2}) is $\sqrt{2(1-q)}\Vert T \Vert= \sqrt{2(1-q)}$. Also, we know that 
    $\omega(T)=\frac{1}{\sqrt{2}}$. Therefore, the left side of inequality (\ref{equ2}) is
    $ \vert \omega_{q}(T)-\omega(T)\vert=\frac{1}{8}\sqrt{27+18q-13q^2+(9+7q)\sqrt{(1-q)(9+7q)}}-\frac{1}{\sqrt{2}}.$ Figure \ref{Fig-4} compares  $ \vert \omega_{q}(T)-\omega(T)\vert$ and the upper bound $\sqrt{2(1-q)}$.  
    
   \begin{figure}[H]
\centering
{\includegraphics[scale=.58]{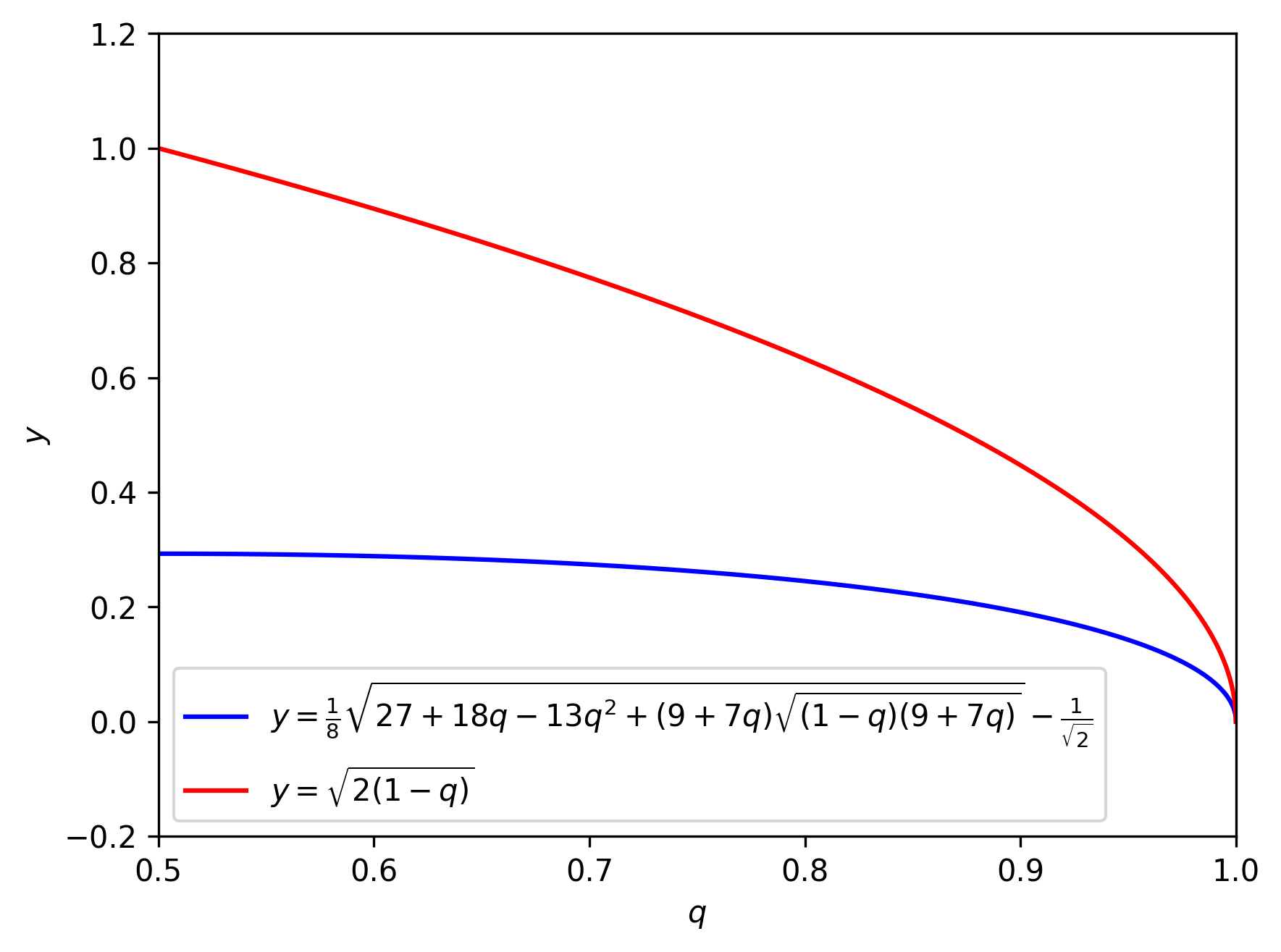}}
\caption{Comparision of $\omega_q(T) $ with the upper bound \eqref{equ2} for Example \ref{Ex_4}. }	 
\label{Fig-4}
\end{figure}
\end{example}

\begin{theorem}
\label{thm5}
For any $ T\in L^{A}(H) $ and $ q\in \mathbb{C}, \ 0<\vert q \vert\leq 1, $ the following are true: \\
(1) $ \vert c_{A,q}(T)-c_{A}(T)\vert\leq \sqrt{2(1-Req)}\Vert T \Vert_{A}. $ \\
(2) If $ q_{n}\in \mathbb{C}, \ 0<\vert q_{n} \vert \leq 1 $ and $ Re q_{n}\rightarrow 1, \ n\rightarrow \infty, $ then 
$$
\lim\limits_{n\rightarrow \infty}c_{A,q_{n}}(T)=c_{A}(T).
$$
(3) For any $ T,S\in L^{A}(H) $ and $ q\in \mathbb{C}, \ 0<\vert q \vert\leq 1, $ we have
$$ 
\vert c_{A,q}(T)-c_{A,q}(S)\vert\leq \omega_{q}(T\pm S).
$$
(4) If the operator sequence $ (T_{n})$ converges uniformly to an operator $ T $ in $ L^{A}(H), $ then
$$
\lim\limits_{n\rightarrow \infty}c_{A,q}(T_{n})=c_{A,q}(T).
$$
\end{theorem}
\begin{proof}
 (1) For $ x, y \in H, $ $ \Vert x \Vert_{A}=\Vert y \Vert_{A}=1, \ \langle x,y \rangle_{A}=q, $ we have 
\begin{eqnarray}
\vert \langle Tx,y \rangle_{A}\vert & \geq & \vert \langle Tx,x \rangle_{A}+\langle Tx,y-x \rangle_{A}\vert  \nonumber \\
& = & \vert \langle Tx,x \rangle_{A} \vert - \vert \langle Tx,y-x \rangle_{A}\vert \nonumber \\
& = & \vert \langle Tx,x \rangle_{A} \vert - \vert \langle ATx,y-x \rangle\vert \nonumber \\
& \geq & \vert \langle Tx,x \rangle_{A} \vert - \vert \langle ATx,Tx \rangle\vert^{1/2}\langle A(y-x),y-x\rangle^{1/2} \nonumber \\
& = & \vert \langle Tx,x \rangle_{A} \vert - \Vert Tx \Vert_{A}\langle y-x,y-x\rangle^{1/2}_{A} \nonumber \\
& \geq & \vert \langle Tx,x \rangle_{A} \vert - \Vert T \Vert_{A} \sqrt{2}\sqrt{1-Req}. \nonumber
\end{eqnarray}
From this and from the definition of the Crawford number, we have
$$
c_{A,q}(T)\geq c_{A}(T)-\Vert T \Vert_{A} \sqrt{2}\sqrt{1-Req},
$$
that is,
$$
c_{A,q}(T)-c_{A}(T)\geq - \Vert T \Vert_{A} \sqrt{2}\sqrt{1-Req}.
$$
Similarly, by the previous method, it can be shown that 
$$
c_{A,q}(T)-c_{A}(T)\leq\Vert T \Vert_{A} \sqrt{2}\sqrt{1-Req}.
$$
Consequently, it is true
$$ 
\vert c_{A,q}(T)-c_{A}(T)\vert\leq \sqrt{2(1-Req)}\Vert T \Vert_{A}.
$$
(2) The validity of this result is clear from the first claim. \\
(3) For any $ x, y \in H, $ $ \Vert x \Vert_{A}=\Vert y \Vert_{A}=1, \ \langle x,y \rangle_{A}=q, \ q\in \mathbb{C}, \ 0<\vert q \vert\leq 1 $, we have
\begin{eqnarray}
\vert \langle Tx,y \rangle_{A} \vert & = & \vert \langle (T+S)x,y \rangle_{A} - \langle Sx,y \rangle_{A}  \vert \nonumber \\
& \geq & \vert \langle (T+S)x,y \rangle_{A}\vert - \vert\langle Sx,y \rangle_{A}  \vert . \nonumber
\end{eqnarray}
From the last relation, we have
$$
c_{A,q}(T)\geq c_{A,q}(T+S)-\omega_{A,q}(S),
$$
that is,
$$
c_{A,q}(T+S)-c_{A,q}(T)\leq\omega_{A,q}(S).
$$
In a similar way, it can be established that 
$$
-\omega_{A,q}(S)\leq c_{A,q}(T+S)-c_{A,q}(T).
$$
Consequently, from the previous results, we have
$$
\vert c_{A,q}(T+S)-c_{A,q}(T) \vert \leq \omega_{A,q}(S).
$$
Hence, it is clear that 
$$ 
\vert c_{A,q}(T)-c_{A,q}(S)\vert\leq \omega_{A,q}(T-S).
$$
(4) The validity of this result follows from the last claim and the inequality
$$
\omega_{A,q}(T_{n}-T)\leq \Vert T_{n}-T \Vert_{A}, \ n\geq 1.
$$
(See the first claim of Theorem \ref{thm1}.) 
\end{proof}
\begin{corollary}
\label{cor2}
For the operator sequence $ (T_{n}) $ in $ L^{A}(H), $ which  converges uniformly to an operator $ T\in L^{A}(H), $ the following are true:
$$
\lim\limits_{n\rightarrow\infty}g_{\omega_{A,q}}(T_{n})=g_{\omega_{A,q}}(T)
$$
and 
$$
\lim\limits_{n\rightarrow\infty}g_{c_{A,q}}(T_{n})=g_{c_{A,q}}(T) .
$$
\end{corollary}
\begin{proof}
Indeed, for the gaps, we have
\begin{eqnarray}
\vert g_{\omega_{A,q}}(T_{n})-g_{\omega_{A,q}}(T)\vert & = & \vert \left( \Vert T_{n} \Vert_{A}- \omega_{A,q}(T_{n})\right) -\left( \Vert T \Vert_{A}-\omega_{A,q}(T)\right) \vert  \nonumber \\
& \leq & \Vert T_{n}-T \Vert_{A} + \vert \omega_{A,q}(T_{n}) -\omega_{A,q}(T)\vert \nonumber \\
& \leq & 2 \Vert T_{n}-T \Vert_{A}, \ n \geq 1 \nonumber
\end{eqnarray}
and 
\begin{eqnarray}
\vert g_{c_{{A,q}}}(T_{n})-g_{c_{{A,q}}}(T)\vert 
& \leq & \Vert T_{n}-T \Vert_{A} + \vert c_{A,q}(T_{n}) -c_{A,q}(T)\vert \nonumber \\
& \leq & 2 \Vert T_{n}-T \Vert_{A}, \ n \geq 1. \nonumber 
\end{eqnarray}
Since $ \Vert T_{n}-T \Vert_{A}\rightarrow 0, \ n\rightarrow\infty, $ then the validity of the corollary is clear.
\end{proof}
%%% ----------------------------------------------------------------------
\setcounter{section}{2}
\section*{\normalsize{4. Some applications}}
In this section, some applications of our results will be given.  \\ \\
\noindent {\it \textbf{Application 1}}
Let $ H_{1} $ and $ H_{2} $ be two complex Hilbert spaces, $ A_{1}\in L(H_{1}), \  A_{2}\in L(H_{2}), \ A_{1}\geq 0, \ A_{2}\geq 0 $ and $ A=A_{1}\oplus A_{2}. $ Consider the following operator sequences 
$$
T_n= S_{n}\oplus M_{n}, \ T_{n}\in L^{A}\left( H_{1}\oplus H_{2} \right), \ n\geq 1.
$$
Assume that $ S_{n}\rightarrow S$ and $\ M_{n}\rightarrow M $  uniformly for some $ S $ and $ M $ in $ L^{A_{1}}(H_{1}) $ and $ L^{A_{2}}(H_{2}),  $ respectively. Also, let 
$$
T=S\oplus M.
$$
It is clear that $ \Vert T \Vert_{A}=\max \left\lbrace \Vert S \Vert_{A_{1}}, \Vert M \Vert_{A_{2}} \right\rbrace . $ \\
For any $ q\in \mathbb{C}, \ \vert q \vert \leq 1, $ it is clear that 
\begin{eqnarray}
\label{equ3}
W_{A_{1},q}(S)\subset W_{A,q}(T), \ W_{A_{2},q}(M)\subset W_{A,q}(T).
\end{eqnarray}
Therefore, $ \max \left\lbrace \omega_{A_{1},q}(S), \omega_{A_{2},q}(M) \right\rbrace \leq \omega_{A,q}(T).  $ Hence, it is easy to see that 
\begin{align}
 &g_{\omega_{A,q}}(T)  =  \Vert T \Vert_{A}-\omega_{A,q}(T)\nonumber \\
& \leq \max \left\lbrace \Vert S \Vert_{A_{1}}, \Vert M \Vert_{A_{2}} \right\rbrace - \max \left\lbrace \omega_{A_{1},q}(S), \omega_{A_{2},q}(M) \right\rbrace  \nonumber \\
& =  \left( \frac{\Vert S \Vert_{A_{1}}+\Vert M \Vert_{A_{2}}}{2} + \vert \frac{\Vert S \Vert_{A_{1}}-\Vert M \Vert_{A_{2}}}{2} \vert \right) - \left( \frac{\omega_{A_{1},q}(S)+ \omega_{A_{2},q}(M)}{2} + \vert \frac{\omega_{A_{1},q}(S)- \omega_{A_{2},q}(M)}{2} \vert \right) \nonumber \\
& \leq  \frac{\Vert S \Vert_{A_{1}} - \omega_{A_{1},q}(S)}{2}+ \frac{\Vert M \Vert_{A_{2}} - \omega_{A_{2},q}(M)}{2}+ \vert  \frac{\Vert S \Vert_{A_{1}}-\Vert M \Vert_{A_{2}}}{2} - \frac{\omega_{A_{1},q}(S)- \omega_{A_{2},q}(M)}{2} \vert \nonumber \\
& =  \frac{g_{\omega_{A_{1},q}}(S)+g_{\omega_{A_{2},q}}(M)}{2} + \vert \frac{g_{\omega_{A_{1},q}}(S)-g_{\omega_{A_{2},q}}(M)}{2} \vert \nonumber \\
& =  \max \left\lbrace g_{\omega_{A_{1},q}}(S),g_{\omega_{A_{2},q}}(M)  \right\rbrace  \nonumber   
\end{align}
for any $ q\in \mathbb{C}, \ \vert q \vert \leq 1. $
That is, for any $ q\in \mathbb{C}, \ \vert q \vert \leq 1 $, it is established that 
$$
g_{\omega_{A,q}}(T) \leq \max \left\lbrace g_{\omega_{A_{1},q}}(S),g_{\omega_{A_{2},q}}(M)  \right\rbrace .
$$
Then, from the Corollary \ref{cor2}, we have
$$
\lim\limits_{n\rightarrow \infty}g_{\omega_{A,q}}(T_{n}) \leq \max \left\lbrace g_{\omega_{A_{1},q}}(S),g_{\omega_{A_{2},q}}(M)  \right\rbrace .
$$
On the other hand, from (\ref{equ3}), it is true that
$$
c_{{A,q}}(T) \leq inf \left\lbrace c_{{A_{1},q}}(S),c_{{A_{2},q}}(M)  \right\rbrace 
$$
for any $ q\in \mathbb{C}, \ \vert q \vert \leq 1. $ Then,
\begin{align}
   & g_{c_{A,q}}(T)  =  \Vert T \Vert_{A}-c_{A,q}(T)\nonumber \\
& \geq  \max \left\lbrace \Vert S \Vert_{A_{1}}, \Vert M \Vert_{A_{2}} \right\rbrace - \inf \left\lbrace c_{A_{1},q}(S), c_{A_{2},q}(M) \right\rbrace  \nonumber \\
& =  \left( \frac{\Vert S \Vert_{A_{1}}+\Vert M \Vert_{A_{2}}}{2} + \vert \frac{\Vert S \Vert_{A_{1}}-\Vert M \Vert_{A_{2}}}{2} \vert \right) - \left( \frac{c_{A_{1},q}(S)+ c_{A_{2},q}(M)}{2} - \vert \frac{c_{A_{1},q}(S)- c_{A_{2},q}(M)}{2} \vert \right) \nonumber \\
& \geq  \frac{\Vert S \Vert_{A_{1}} - c_{A_{1},q}(S)}{2}+ \frac{\Vert M \Vert_{A_{2}} - c_{A_{2},q}(M)}{2}+ \vert  \frac{\Vert S \Vert_{A_{1}}-\Vert M \Vert_{A_{2}}}{2} - \frac{c_{A_{1},q}(S)- c_{A_{2},q}(M)}{2} \vert \nonumber \\
& =  \frac{g_{c_{A_{1},q}}(S)+g_{c_{A_{2},q}}(M)}{2} + \vert \frac{g_{c_{A_{1},q}}(S)-g_{c_{A_{2},q}}(M)}{2} \vert \nonumber \\
& =  \max \left\lbrace g_{c_{A_{1},q}}(S),g_{c_{A_{2},q}}(M)  \right\rbrace  \nonumber .
\end{align}

Then, by Corollary \ref{cor2}, we have
$$
\lim\limits_{n\rightarrow \infty}g_{c_{A,q}}(T_{n}) \geq \max \left\lbrace g_{c_{A_{1},q}}(S),g_{c_{A_{2},q}}(M) \right\rbrace 
$$
for any $ q\in \mathbb{C}, \ \vert q \vert \leq 1. $ \\ 
\noindent {\it \textbf{Application 2}}
Let $ \psi \in C[0,1], \ \psi \geq 0, \ \varphi_{n}\in C[0,1], \ n\geq 1, \ \varphi_{n}\rightarrow 1  $ uniformly on $ [0,1] $ and 
$$
Af(x)=\psi(x)f(x), \ f\in L_{2}(0,1),
$$
$$
A: L_{2}(0,1) \rightarrow L_{2}(0,1),
$$
$$
T_{n}f(x)=\varphi_{n}(x)f(x), \ f\in L_{2}^{\psi}(0,1),
$$
where,
$$
L_{2}^{\psi}(0,1)= \left\lbrace f:(0,1)\rightarrow \mathbb{C}: f \ \text{Lebesque  measurable function and}\ \int\limits_{0}^{1}\psi (x) \vert f(x) \vert^{2}dx<\infty \right\rbrace.  
$$ 
Then, the operator sequences $ T_{n}, \ n\geq 1, $ converges uniformly to the identity operator $ I $ in $ L_{2}^{\psi}(0,1). $ Hence, for any $ q\in \mathbb{C}, \ \vert q \vert \leq 1, $ by Corollary \ref{cor2}, we have
$$
\lim\limits_{n\rightarrow \infty}g_{\omega_{A,q}}(T_{n})=g_{\omega_{A,q}}(T)=1-\vert q \vert,
$$
$$
\lim\limits_{n\rightarrow \infty}g_{c_{A,q}}(T_{n})=g_{c_{A,q}}(T)=1-\vert q \vert .
$$
\\
\noindent {\it \textbf{Application 3}}
A prominent function class $ (\Lambda_{\theta})_{+} $ will be given (see \cite{Aleksandrov}). Let $ \theta $ be a modulus of continuity, i. e., $ \theta $ is a non-decreasing continuous function on $ [0,\infty) $ such that $ \theta (0)=0, \ \theta(x)>0 $ for $ x>0, $ and 
$$
\theta (x+y)\leq \theta (x)+ \theta(y), \ x,y\in [0,\infty).
$$
In this case, $ (\Lambda_{\theta})_{+} $ is defined as
$$
(\Lambda_{\theta})_{+} := \left\lbrace  f\in A(\mathbb{D}): \vert f \vert_{\Lambda_{\theta}}= \sup_{\substack{a,b\in\mathbb{D} \\ a\neq b}}\frac{\vert f(a)-f(b)\vert}{\theta(\vert a-b\vert)}<\infty\right\rbrace ,
$$
where $ \mathbb{D}:= \lbrace z\in \mathbb{C}: \vert z \vert<1 \rbrace $ is the unit disc and $ A(\mathbb{D}) $ is the class of analytic functions on $ \mathbb{D}. $ \\
Also, when $ \theta $ is given, $ \theta_{*} $ will be defined as
$$
\theta_{*}(x):=x \int\limits_{x}^{\infty}\frac{\theta(t)}{t^{2}}dt, \ x>0.
$$
Remark that $ \lim\limits_{x\rightarrow 0^{+}}\theta_{*}(x)=0. $
\begin{theorem} 
\cite{Aleksandrov} For every modulus of continuity $ \theta, $ for arbitrary contractions $ S $ and $ T, $ and $ f\in (\Lambda_{\theta})_{+}, $ there exists a constant $ c>0 $ such that 
$$
\Vert f(S)-f(T)\Vert\leq c \vert f \vert_{\Lambda_{\theta}}\theta_{*}(\Vert S-T \Vert).
$$
\end{theorem}
Now, we will examine how the results in this section vary for operator functions.
\begin{theorem}
Let $ (T_{n}), \ n\geq 1 $ in $ L^{A}(H) $ be a sequence of contraction operators, which converges uniformly to an operator $ T\in L^{A}(H) $ and $ f\in (\Lambda_{\theta})_{+}. $ In this case,
$$
\omega_{A,q}(f(T))=\lim\limits_{n\rightarrow\infty}\omega_{A,q}(f(T_{n})), \ q\in \mathbb{C}, \ 0<\vert q \vert\leq 1,
$$
$$
c_{A,q}(f(T))=\lim\limits_{n\rightarrow\infty}c_{A,q}(f(T_{n})), \ q\in \mathbb{C}, \ 0<\vert q \vert\leq 1
$$
are true.
\end{theorem}
\begin{proof}
Let $ (T_{n}), \ n\geq 1 $ in $ L^{A}(H) $ be a sequence of contraction operators, which  converges uniformly to an operator $ T\in L^{A}(H) $ and $ f\in (\Lambda_{\theta})_{+}. $ In this situation, $ T $ is a contraction operator. Also, since $ f\in (\Lambda_{\theta})_{+}, $ then there exists $ c>0 $ such that $ \Vert f(T)-f(T_{n})\Vert\leq c \vert f \vert_{\Lambda_{\theta}}\theta_{*}(\Vert T-T_{n} \Vert), \ n\geq 1, $ by Theorem \ref{thm3}. Thus, since $ \lim\limits_{x\rightarrow 0^{+}}\theta_{*}(\Vert T-T_{n} \Vert)=0, $ then the operator sequence $ (f(T_{n})) $  converges uniformly to $ f(T). $ Therefore, the validity of the assertions of this theorem under the corresponding conditions is obvious from Theorems \ref{thm4} and \ref{thm5}. 
\end{proof}

\textbf{Conflict of interest:} The authors state that there is no conflict of interest.

\textbf{Data availability:} Data sharing not applicable to the present paper as no
data sets were generated or analyzed during the current study.

\textbf{Funding:} Not applicable.


\begin{thebibliography}{16}

\bibitem{Aleksandrov} Aleksandrov, A. B., Peller, V. V., \textit{Operator Hölder-Zygmund functions}, Adv. Math. {\bf 224} (2010), 910-966.

\bibitem{baklouti} Baklouti, H., Feki, K., Sid Ahmed, O. A. M., \textit{Joint numerical ranges of operators in semi-Hilbertian spaces}, Linear Algebra Appl. {\bf 555}  (2018), 266-284.

\bibitem{bhunia4} Bhunia, P., Dragomir, S. S., Moslehian, M. S., Paul, K., \textit{Lectures on numerical radius inequalities}, Infosys Science Foundation Series. Cham: Springer. 2022.


\bibitem{bhunia5} Bhunia, P., Feki, K., Paul, K., \textit{Numerical radius inequalities for products and sums of semi-Hilbertian space operators}, Filomat {\bf 36} (4) (2022), 1415-1431.

\bibitem{bhunia6} Bhunia, P., Nayak, R. K., Kallol, P., \textit{Refinement of seminorm and numerical radius inequalities of semi-Hilbertian space operators}, Math. Slovaca  {\bf 72}(4) (2022), 969-976.

\bibitem{bhunia1} Bhunia, P., Paul, K. \textit{Refinements of norm and numerical radius inequalities}, Rocky Mountain J. Math. {\bf 51} (6) (2021), 1953-1965.

\bibitem{bhunia2} Bhunia, P., Paul, K., \textit{Proper improvement of well-known numerical radius inequalities and their applications}, Result Math. {\bf 76} (2021), 1-12.


\bibitem{bhunia3} Bhunia, P., Paul, K., \textit{New upper bounds for the numerical radius of Hilbert space operators}, Bull. Sci. Math. {\bf 167} (2021), 1-11.


\bibitem{conde} Conde, C., Feki, K., \textit{On some inequalities for the generalized joint numerical radius of semi-Hilbert space operators}, Ricerche di Matematica (2021), 1-19. 

\bibitem{FS} Feki, K., Sahoo, S., \textit{Further inequalities for the ${A}$-numerical radius of certain $2 \times 2$ operator matrices}, {Georgian Math. J.} {\textbf { 30}} (2023), no. 2, 213–226.


\bibitem{GAu}Gau, Hwa-Long, and Pei Yuan Wu., \textit{Numerical ranges of Hilbert space operators}, {\bf179} Cambridge University Press, 2021.

\bibitem{KEGustafsonDKMRao1991}
Gustafson, K. E. and Rao, D. K. M., \textit{Numerical Range. The field of values of linear operators and matrices}, {Springer-Verlag}, New York, 1997.


\bibitem{heydarbeygi} Heydarbeygi, Z., Amyari, M., Khanehgir, M. \textit{Some refinements of numerical radius inequalities}, Ukr. Math. J. {\bf 72} (2021), 1664-1674.

\bibitem{kittaneh} Kittaneh, F., \textit{Numerical radius inequalities for Hilbert space operators}, Studia Math. {\bf 168}(1) (2005), 73-80.

 \bibitem{KITSAT} Kittaneh, F., Sahoo, S. \textit{On $ \mathbb{A}$-numerical radius equalities and inequalities for certain operator matrices}, Ann. Funct. Anal. \textbf{12} (52) (2021) 
 {https://doi.org/10.1007/s43034-021-00137-6}
 
\bibitem{li} L{\i}, C. K., Mehta, P. P., Rodman, L.,  \textit{A generalized numerical range: the range of a constrained sesquilinear form}, Linear Multilinear Algebra {\bf 37} (1994), 25-49.

\bibitem{LiNakazato} Li, C. K, Nakazato, H., \textit{Some results on the $q$-numerical}, Linear Multilinear Algebra
{\bf 43} (1998), 385--409.

\bibitem{marcus} Marcus, M., Andresen, P., \textit{ Costrained extrema of bilinear functionals}, Mon. Hefte Math. {\bf 84} (1977), 219-235.

 \bibitem{MMJ} Moghaddam, S. F., Mirmostafaee, A. K., Janfada, M., \textit{$q$-numerical radius inequalities for Hilbert space}, Linear Multilinear Algebra {\bf 72} (5) (2024), 751--763.

\bibitem{Nakazato} Nakazato, H., \textit{The $C$-numerical range of a $2\times2$ matrix}, Sci Rep Hirosaki Univ. {\bf 41} (1994), 197--206.

\bibitem{PR} Patra, A., Roy, F., \textit{On the estimation of $q$-numerical radius of Hilbert space operators}, Oper. Matrices {\bf 18} (2024), 343--359.

\bibitem{qiao} Qiao, H., Hai, G., Bai, E.,  \textit{Some refinements of numerical radius for $ 2\times 2 $ operators matrices}, J. Math. Inequal. {\bf 16} (2) (2022), 425-444.

\bibitem{rajic} Rajic, R., \textit{A generalized $ q$-numerical range}, Pac. J. Math. {\bf 10} (2005), 31-45.
\bibitem{NSD} Rout, N. C., Sahoo, S., Mishra, D., \textit{Some $A$-numerical radius inequalities for semi-Hilbertian space operators}, Linear  Multilinear Algebra {\bf 69} (2021), 980--996.
		
		\bibitem{Nirmal2} Rout, N. C., Sahoo, S., Mishra, D.,
		\textit{On $\mathbb{A}$-numerical radius inequalities for $2\times2$ operator matrices}, Linear  Multilinear Algebra \textbf{70}(14) (2022), 2672--2692.

\bibitem{saddi} Saddi, A., \textit{$A$-normal operators in Semi-Hilbertian space}, Aust. J. Math. Anal. Appl. {\bf 9} (1) (2012), 1-12.
\bibitem{SS}  Sahoo, S., \textit{On $\mathbb{A}$-numerical radius inequalities for $2\times2$ operator matrices-II}, Filomat {\textbf {35}} (15) (2021), 5237--5252.
\bibitem{sheybani} Sheybani, S., Sababheh, M., Moradi, H. R., \textit{Weighted inequalities for the numerical radius}, Vietnam J. Math. {\bf51} (2023), 363--377.


\bibitem{stampfli} Stampfli, J. G., \textit{The norm of derivation}, Math. Commun. {\bf 33} (3) (1970), 737--747.


	
  
\bibitem{Zam} Zamani, A.,  \textit{A-numerical radius inequalities for semi-Hilbertian space operators,} Linear Algebra Appl. \textbf{578} (2019), 159--183.

 










\end{thebibliography}
\end{document}